\documentclass[12pt,a4paper,reqno]{amsart}
\allowdisplaybreaks	
\usepackage[top = 1in, bottom = 1in, left=1.3in, right=1.3in]{geometry}

\usepackage{xcolor}
\usepackage{hyperref}
\hypersetup{colorlinks, linkcolor={red!50!black}, citecolor={green!50!black}, urlcolor={blue!50!black}}

\usepackage[initials]{amsrefs}
\usepackage{amsmath, amssymb}
\usepackage[foot]{amsaddr}  
\usepackage[noabbrev, capitalise]{cleveref}

\theoremstyle{plain}
\newtheorem{theorem}{Theorem}
\newtheorem{proposition}[theorem]{Proposition}
\newtheorem{corollary}[theorem]{Corollary}

\newtheorem{conjecture}[theorem]{Conjecture}
\newtheorem{claim}[theorem]{Claim}

\theoremstyle{remark}
\newtheorem*{remark}{Remark}

\theoremstyle{definition}

\newtheorem{question}[theorem]{Question}

\def\le{\leqslant}
\def\leq{\leqslant}
\def\ge{\geqslant}
\def\geq{\geqslant}
\def\eps{\varepsilon}
\DeclareMathOperator{\probability}{Pr}
\DeclareMathOperator{\expectation}{E}
\DeclareMathOperator{\variance}{Var}
\DeclareMathOperator{\covariance}{Cov}

\begin{document} 
	
	\author[Pavez-Sign\'e]{Mat\'ias Pavez-Sign\'e$^1$}
	\author[Quiroz]{Daniel A. Quiroz$^2$}
	\address{$^1$:  Center for Mathematical Modeling (CNRS-IRL 2807), Universidad de Chile, Beauchef 851, Santiago, Chile.}
	\address{$^2$: Instituto de Ingenier\'ia Matem\'atica-CIMFAV, Universidad de Valparaiso, General Cruz 222, Valparaiso, Chile.}
	\author[Sanhueza-Matamala]{Nicol\'as Sanhueza-Matamala$^3$}
	\address{$^3$: The Czech Academy of Sciences, Institute of Computer Science, Pod Vod\'{a}renskou v\v{e}\v{z}\'{\i} 2, 182 07 Prague, Czech Republic.}
	
	\title{Universal arrays}
	
	\begin{abstract}
		A word on $q$ symbols is a sequence of letters from a fixed alphabet of size~$q$.
		For an integer $k\ge 1$, we say that a word $w$ is $k$-universal if, given an arbitrary word of length $k$, one can obtain it by removing letters from $w$.
		It is easily seen that the minimum length of a $k$-universal word on $q$ symbols is exactly $qk$.
		We prove that almost every word of size $(1+o(1))c_qk$ is $k$-universal with high probability, where $c_q$ is an explicit constant whose value is roughly $q\log q$.
		Moreover, we show that the $k$-universality property for uniformly chosen words exhibits a sharp threshold.
		Finally, by extending techniques of Alon [Geometric and Functional Analysis   27 (2017), no. 1, 1--32], we give asymptotically tight bounds for every higher dimensional analogue of this problem.
	\end{abstract}
	
	\maketitle
	
	\section{Introduction}
	
	
	A \emph{universal} mathematical structure is one which contains all possible substructures of a particular form. Famous examples of universal structures are De Bruijn sequences~\cite{DeBruijn}, which are periodic words that contain, exactly once, every possible word of a fixed size as a substring. Universal structures where perhaps first considered in a general sense by Rado~\cite{Rado1964}, who studied the existence of universal graphs, hypergraphs and functions for various notions of containment.
	
	The study of universal (finite) graphs has received particular attention, and for these the containment relation of choice has been that of induced subgraphs.
	Thus a graph $G$ is said to be \emph{$k$-universal} if $G$ contains every graph on $k$ vertices as an induced subgraph.
	Two problems have been at the centre of the study of $k$-universal graphs.
	The first one is to determine $n$, the minimum value such that there exists a $k$-universal graph on $n$ vertices.
	In 1965, Moon~\cite{moon_1965} gave, through a simple counting argument, a lower bound of $2^{(k-1)/2}$ for that value of $n$.
	Recently, Alon~\cite{alon2017} showed that this lower bound is asymptotically tight, essentially settling this 50-year-old problem.
	More so, in a later paper, Alon and Sherman~\cite{alon2019} gave an asymptotically tight bound for the hypergraph generalisation of this problem.
	The second central problem in the study of $k$-universal graphs is the ``random'' analogue of the previous question, that is, finding the minimum $n$ such that ``almost every" $n$-vertex graph is $k$-universal.
	After works of Bollob\'as and Thomason~\cite{bollobas1981}, and Brightwell and Kohayakawa~\cite{brightwell1993}, Alon~\cite{alon2017} has essentially settled this problem as well.
	
	Finding a $k$-universal graph is equivalent to finding an adjacency matrix which ``contains'' the adjacency matrices of all $k$-vertex graphs. Here we are considering that an adjacency matrix $M$ contains another matrix $M'$, if we can obtain $M'$ from $M$ by iteratively applying the following operation: choose a value $i$ and delete the $i$-th row and the $i$-th column.
	It is thus natural to consider square matrices together with the notion of containment given by the operation of choosing values $i,j$ and deleting row~$i$ and column~$j$, and its associated notion of universality.
	More generally, we shall consider the analogue of this notion of containment for ``$d$-dimensional arrays'' for all $d\ge 1$.
	
	In what follows, we use the common notation $[k]=\{1, \dotsc ,k\}$, for any integer $k\ge 1$.
	Given an alphabet $\Sigma$, a positive integer $d$, and a $d$-tuple $\mathbf{n} = (n_1, \dotsc, n_d) \in \mathbb{N}^d$,
	a \emph{$d$-dimensional array of size $\mathbf{n}$ over $\Sigma$} is a collection of symbols $a_{\mathbf{i}} \in \Sigma$ indexed by the vectors $\mathbf{i} \in [n_1] \times [n_2] \times \dotsb \times [n_d]$.
	With regard to the alphabet $\Sigma$ we are only interested in its cardinality, and will assume $\Sigma=[q]$, whenever $|\Sigma|=q$. Thus $\Sigma$ will usually be clear from the context and, for short, we will just talk about \emph{$d$-arrays} of a certain size. 
	A $d$-array of \emph{order $n$} is a $d$-array of size $(n, n, \dotsc, n)$.
	Note that $1$-arrays of size $n$ are commonly referred to as \emph{words} of \emph{length~$n$}.
	
	For general $d$-arrays and a corresponding notion of universality, we study the analogue of the two questions settled by Alon in the graph case (the ``deterministic'' and ``random'' questions).
	Whenever $d \ge 2$, we obtain asymptotically tight bounds for both questions (see Theorem~\ref{thm:2} and Corollary~\ref{coro}, below) by extending a method used by Alon in the graph case.
	However, this technique does not seem (directly) to work when $d=1$, that is, for the case of words.
	For this case we develop different tools which allow us to show tight bounds for both problems (see Theorems~\ref{theorem:universalwords-deterministic} and~\ref{theorem:universalwords-random}). 
	
	Let us first define the notion of containment we will consider for general $d$-arrays, which is a generalisation of the containment notion for matrices discussed above.
	For fixed $d$, let $A = (a_{\mathbf{i}})$ be a $d$-array of size $(n_1, \dotsc, n_d)$.
	We define the \emph{coordinate restriction operation} on $A$ as follows.
	Choose some $j \in [d]$ and $\ell \in [n_j]$.
	Delete all the symbols whose $j$-th coordinate is $\ell$, to obtain a $d$-array of size $(n_1, n_2,  \dotsc, n_{j-1}, n_j - 1, n_{j+1}, \dotsc, n_d)$.
	We say a $d$-array $A$ \emph{contains} a $d$-array $A'$ if we can obtain $A'$ by iteratively applying coordinate restriction operations, and consider universal $d$-arrays under this containment notion. 
	
	For fixed $d,k\ge 1$, and a fixed alphabet $\Sigma$, we say a $d$-array over~$\Sigma$  is \emph{$k$-universal} if it contains every $d$-array $A$ on $\Sigma$ of size $(n_1, n_2, \dotsc , n_d)$, where $n_j \le k$ for all $j\in [d]$.
	Note that if we want to show that a given $d$-array is $k$-universal, it is enough to show that it contains every $d$-array of order $k$.
	We let $f_{d}(q,k)$ be the minimum $n$ such that there exists a $k$-universal $d$-array of order $n$ over the $q$-symbol alphabet. 

	
	It is important to distinguish the notion of containment we consider for the case of words from that considered in De Bruijn sequences.
	We obtain a smaller word from a larger one by deleting entries (this notion is commonly called \emph{subword} or \emph{subsequence}), while De Bruijn (cyclic) sequences contain each smaller word in a contiguous manner (a notion usually called \emph{substring} or \emph{factor}).
	In particular, a substring is always a subword but not vice versa.
	From a combinatorial perspective, the notion of subword is just as natural, and has been considered in various extremal problems.
	Examples include the ``twins problem"~\cite{APP,bukh2016twins,bukh2014longest} which asks for two disjoint identical subwords of a given word which leave as few unused symbols as possible.
	Another example is the ``longest common subsequence problem"\cite{chvatal, kiwi2005expected}, which seeks the expected value of the longest common subsequence of two randomly chosen words.
	The notion of subword also has additional desirable properties.
	For instance, the subword density of a fixed length word is continuous with respect to the cut-distance in the limit theory for words sequences~\cite{HKPS}.
	
	Our results in the case of words are the following. 
	
	\begin{theorem} \label{theorem:universalwords-deterministic}
		Let $k \ge 1$ and $q \ge 2$ be integers. Then $f_1(q,k) = qk$.
	\end{theorem}
	
	This result establishes the gap between the notions of subword and substring.
	While a minimal $k$-universal word has size $qk$, a De Bruijn sequence has size $\Theta(q^k)$.
	
	We also obtain the following ``threshold'' behaviour for randomly chosen words to be $k$-universal.
	For any $q \geq 1$, let $H_q = \sum_{i=1}^q 1/i$ be the $q$th harmonic number.
	
	\begin{theorem} \label{theorem:universalwords-random}
		Let $q \ge 2$ be a fixed integer and 
		$c_q = q H_q$.
		Consider a uniformly chosen word $w$ of length $n = n(k)$ over the $q$-symbol alphabet.
		For every $\eps > 0$ we have
		\[ \probability[\,\text{$w$ is $k$-universal}\:] \rightarrow
		\begin{cases}
			0 & \text{if } n \leqslant (1 - \eps) c_q k, \text{ and} \\
			1 & \text{if } n \geqslant (1 + \eps) c_q k,
		\end{cases} \]
		where the limit is taken as $k \rightarrow \infty$.
	\end{theorem}
	
	In particular, for the $2$-symbol alphabet, we have $f_1(2,k) = 2k$, while roughly $3k$ symbols are necessary and sufficient for a typical binary word of that length to be $k$-universal.
	This last statement answers a question of Biers-Ariel, Godbole and Kelley~\cite{BGK18}.
	
	The following theorem and its corollary are our results for general $d$-arrays with $d\ge 2$.
	
	\begin{theorem}\label{thm:2}
		Let $d, q \ge 2$ be fixed integers. For every $\eps>0$, a uniformly chosen $d$-array  of order $n=(1+\eps)\frac keq^{ \frac{k^{d-1}}{d}}$ over the $q$-symbol alphabet is $k$-universal with high probability as $k\to\infty$.
	\end{theorem}
	
	Furthermore, a simple counting argument gives $f_d(q, k) \geqslant \frac{k}{e} q^{ \frac{k^{d-1}}{d}}$ (see Section 3).
	Thus we obtain the following.
	\begin{corollary}\label{coro}
		Let $d, q \ge 2$ be fixed integers. We have $f_d(q, k)= (1 + o_k(1)) \frac{k}{e} q^{ \frac{k^{d-1}}{d}}.$
	\end{corollary}

	We point out that the cases  $d=1$ and $d\ge 2$ behave in completely different manners.
	In the case $d=1$, the case of words, the value of $n$ in the random version is considerably larger than $f_1(q,k)$ (a similar scenario holds for the graph case~\cite{alon2017}). In contrast, for $d$-arrays with $d\ge2$ the order which is necessary for random $d$-arrays to be $k$-universal is asymptotically equal to $f_d(q,k)$.
	
	The paper is organised as follows. The proofs of Theorem~\ref{theorem:universalwords-deterministic} and \ref{theorem:universalwords-random} are found in Section~\ref{secwords}.
	In Section~\ref{secarrays}, we prove Theorem~\ref{thm:2} and give the counting argument which implies Corollary~\ref{coro}.
	The paper ends with concluding remarks in Section~\ref{seclast}.
	
	\section{Universal words}\label{secwords}
	In this section we prove Theorems~\ref{theorem:universalwords-deterministic} and~\ref{theorem:universalwords-random}.
	We will use $\Sigma = [q]$ as the fixed $q$-symbol alphabet.
	We recall the standard notation used to work with words.
	Given a word $w$ and an integer $k$, $w^k$ is the $k$-fold concatenation of $w$ with itself $k$ times.
	We write $\Sigma^k$ for the set of all words of length $k$ over $\Sigma$ and $\Sigma^\ast$ for the set of all words over $\Sigma$.
	
	Although Theorem~\ref{theorem:universalwords-deterministic} can be proved directly,
	we will derive it using stronger tools, which we will need to prove Theorem~\ref{theorem:universalwords-random}.
	To introduce these tools, we begin with a few definitions.
	Given any word $w$ in $\Sigma^\ast$, define $U_{\Sigma}(w)$ as the minimal prefix of $w$ which contains all symbols of $\Sigma$ if it exists, or $U_{\Sigma}(w) = w$ otherwise.
	Define $T_{\Sigma}(w)$ as $w$ with the prefix $U_{\Sigma}(w)$ removed.
	Given a word $w$, we can greedily decompose it in a unique way as $w = u_1 u_2 \dotsb u_{\ell} u'$ such that for all $i \in [\ell]$, $u_i = U_{\Sigma}( u_i u_{i+1} \dotsb u_\ell u' )$ and $T_{\Sigma}(u_i u_{i+1} \dotsb u_\ell u') = u_{i+1} \dotsb u_{\ell} u'$, each $u_i$ contains all the symbols of $\Sigma$ and $u'$ (possibly empty) does not contain all the symbols of $\Sigma$.
	We say $u_1 u_2 \dotsb u_{\ell} u'$ is the \emph{$\Sigma$-universal decomposition of $w$} and we let $\nu_{\Sigma}(w) = \ell$.
	We can use $\nu_{\Sigma}(w)$ to characterise $k$-universal words, as follows.
	
	\begin{proposition} \label{proposition:universalword-characterisation}
		A word $w \in \Sigma^\ast$ is $k$-universal if and only if $\nu_{\Sigma}(w) \geqslant k$.
	\end{proposition}
	
	\begin{proof}
		Suppose $w$ satisfies $\nu_{\Sigma}(w) \ge k$.
		Then $w$ has as a prefix a substring $u_1 u_2 \dotsb u_k$ where each of the words $u_i$ contains all of the symbols from $\Sigma$.
		Then any word $x \in \Sigma^k$ can be found greedily as a subword in $w$ by finding the $i$-th symbol of $x$ inside the word~$u_i$.
		
		In the other direction, suppose $\nu_{\Sigma}(w) = k' < k$ and let $w = u_1 \dotsb u_{k'} u'$ be the $\Sigma$-universal decomposition of $w$.
		Since each $u_i$ is a minimal prefix of $u_i \dotsb u_{k'} u'$ that contains all the symbols of $\Sigma$, it must have the form $u_i = v_i \sigma_i$, where $\sigma_i$ is a symbol in $\Sigma$ and $v_i$ does not use the symbol $\sigma_i$.
		Further, let $\sigma_{k' + 1}$ be any symbol in $\Sigma$ which does not appear in $u'$ (which exists by definition).
		We claim that $w$ does not contain the word $w' = \sigma_1 \sigma_2 \dotsb \sigma_{k'} \sigma_{k' + 1}$.
		Since $k'+1 \le k$, this readily implies that $w$ is not $k$-universal.
		
		To find a contradiction, suppose that $w'$ is contained in $w$.
		The first symbol of $w'$ is $\sigma_1$, and the first time $\sigma_1$ appears in $w$ is at the end of $u_1$, and thus the remaining symbols must appear after the end of $u_1$.
		That means the word $\sigma_2 \dotsb \sigma_{k'} \sigma_{k' + 1}$ is contained in $u_2 \dotsb u_{k'} u$.
		Using the same reasoning, we see that for all $j \leq k'$, the $j$-th symbol of $w'$ appears in $w$ only after the last symbol of $u_j$.
		Therefore, the last symbol of $w'$, which is $\sigma_{k' + 1}$, appears as a symbol in $u'$, a contradiction.
	\end{proof}
	
	Note that Theorem~\ref{theorem:universalwords-deterministic} follows trivially from Proposition~\ref{proposition:universalword-characterisation}.
	This is because any word containing all symbols from $\Sigma$ must have size at least $|\Sigma| = q$,
	thus a word $w$ with $\nu_\Sigma(w) \ge k$ must have least $qk$ symbols.
	Equality is attained by $(12 \dotsb q)^k$.
	
	Now, we will need to estimate $\nu_{\Sigma}(w)$ for a uniformly chosen random word $w$.
	We will appeal to the well-known ``coupon-collector problem''.
	Given a $q$-sized set $Q$ and a sequence $X_1, X_2, \dotsc$ of independent and uniformly chosen random variables $X_i \in Q$ for all $i \ge 1$, define the random variable $T$ as the minimum integer such that $\{ X_1, \dotsc, X_T \} = Q$.
	It is known that $T$ can be written as the sum of $q$ independent geometric random variables $T = G_1 + \dotsb + G_q$, where $G_j$ has parameter $j/q$ for each $j \in [q]$, and from this it is deduced that $\expectation[T] = c_q := q H_q$.

	Now we are ready for the proof of Theorem~\ref{theorem:universalwords-random}.
	
	\begin{proof}[Proof of Theorem~\ref{theorem:universalwords-random}]
		Let $\Sigma$ be an alphabet of size $q$.
		To estimate $\nu_\Sigma(w)$ of a random word $w$, we will couple $w$ with a word created from ``coupon-collector'' experiments, as follows.
		Define a random string $U \in \Sigma^\ast$ using the following process.
		Initially, let $U=\sigma_0$  be a word of length 1, where $\sigma_0$ is chosen uniformly from $\Sigma$.
		If $U$ already has all the symbols of $\Sigma$, stop.
		Otherwise, choose uniformly and independently a symbol $\sigma \in \Sigma$ and update~$U$ by appending $\sigma$ at the end.
		Clearly, the length $|U|$ of $U$ distributes as the random variable $T$ defined before the start of the proof and thus $\expectation[|U|] = c_q$.
		Given $k > 0$, let $U_1, \dotsc, U_{k}$ be independent random strings, each of them distributed as $U$, and let $U^{(k)} = U_1 U_2 \dotsb U_{k}$ be their concatenation.
		Crucially, we have $\nu_\Sigma(U^{(k)}) = k$, and each strict prefix $u$ of $U^{(k)}$ satisfies $\nu_\Sigma(u) < k$.
		
		Given $k, n > 0$, we construct a (random) word $w$ in $\Sigma^n$ as follows: if $|U^{(k)}| \ge n$ then let $w$ be the first $n$ symbols of $U^{(k)}$; otherwise, construct $w'$ from $U^{(k)}$ by appending $n - |U^{(k)}|$ fresh random symbols at the end of $U^{(k)}$. 
		Note that each symbol of $w$ is chosen independently and uniformly over the symbols of $\Sigma$, so $w$ corresponds exactly to a word on $\Sigma^n$ chosen uniformly at random.
		By construction it is clear that, for all $k, n > 0$,
		\begin{align}
			\probability[\,\text{$w$ is $k$-universal}\:]
			= \probability[ \nu_\Sigma(w) \ge k ]
			= \probability[ |U^{(k)}| \le n ],
			\label{equation:coupling}
		\end{align}
		where the first equality is due to Proposition~\ref{proposition:universalword-characterisation}.
		
		To estimate the last probability, note that $|U^{(k)}| = \sum_{i=1}^\ell |U_i|$ and recall that each of the $|U_i|$  
		has expectation equal to $c_q$.
		Thus, by the (Weak) Law of Large Numbers, we have that, for all $\eps > 0$,
		\begin{align}
			\probability[ (1 - \eps) c_q k \le |U^{(k)}| \le (1 + \eps) c_q k ] \rightarrow 1,
			\label{equation:couponconcentration}
		\end{align}
		whenever $k$ goes to infinity.
		In particular, if $n \leq (1 - \eps)c_q k$ then $\probability[ |U^{(k)}| \le n ] \rightarrow 0$; and if $n \geq (1 + \eps)c_q k$ then $\probability[ |U^{(k)}| \le n ] \rightarrow 1$. By~\eqref{equation:coupling}, the result follows.
	\end{proof}
	
	\begin{remark}
		Theorem~\ref{theorem:universalwords-random} admits an improvement over the size of the error term $\eps$, which can be replaced by any function in $\omega(k^{-1/2})$.
		We sketch the proof.
		We do the same as before, but instead of~\eqref{equation:couponconcentration} one should use that $\Pr[ | |U^{(k)}| - c_q k| > \omega(1) \sqrt{k} ] \rightarrow 0$, where $\omega(1)$ is any function that goes to infinity together with $k$.
		To prove this last statement, write each $|U_i|$ as the sum of $q$ geometric random variables $|U_i| = \sum_{j=1}^q G^{(i)}_j$, where $G^{(i)}_j$ has expectation~$q/j$.
		We want to bound the probability that $|U^{(k)}| - c_q k > \omega(1) \sqrt{k}$ holds (the event given by the ``reverse'' inequality can be treated analogously).
		If the inequality holds, then there exists $j \in [q]$ such that $\sum_{i=1}^k G^{(i)}_j - qk/j > \omega(1) \sqrt{k} / q$.
		A sum of independent geometric random variables follows a negative binomial distribution, which admits a Chernoff--Hoeffding-type deviation bound (see, e.g., \cite[Problem~2.4]{DubhashiPanconesi2009}), and using it gives the result.
	\end{remark}
	
	\section{Universal $d$-arrays}\label{secarrays}
	
	As before, let $\Sigma = [q]$ be the $q$-symbol alphabet.
	For integers $d,k\ge 1$, we write $\mathcal A_d(\Sigma,k)$ for the set of all $d$-arrays of order $k$ over $\Sigma$.
	In this section, we prove Theorem~\ref{thm:2} and establish the lower bound for $f_d(q,k)$ which implies Corollary~\ref{coro}.
	To do so, we first need the following well-known estimates for binomial coefficients, most of which follow from Stirling's approximation. 
	For all $n, k \ge 1$,
	\begin{align}
		k!\ge \left(\frac ke\right)^k\hspace{.5cm}\text{and}\hspace{.5cm}	\binom{n}{k} \le \left( \frac{en}{k} \right)^k. \label{equation:binomialsimpleupperbound}
	\end{align}
	Further, if $k \rightarrow \infty$ as $n\rightarrow \infty$, while $k = o(\sqrt{n})$, 
	\begin{align}
		\binom{n}{k} = (1 + o(1)) \frac{1}{\sqrt{2 \pi k}} \left( \frac{en}{k} \right)^k, \label{equation:stirling}
	\end{align}
	and if $k=\Omega(n)$ and $k \leq n/2$ then 
	\begin{align}
		\log_2\binom{n}{k} = (1 + o(1))H\left( \frac{k}{n} \right)n, \label{equation:entropy}
	\end{align}
	where $H(x)=-x\log_2x-(1-x)\log_2(1-x)$ is the binary entropy.
	
	The lower bound for $f_d(q,k)$ when $d\ge 2$ is given by the following counting argument.
	Notice that there are $q^{k^d}$ $q$-symbol $d$-arrays of order $k$.
	Therefore, a $q$-symbol $d$-array of order $n$ must satisfy
	\[\binom{n}{k}^d\ge q^{k^d}\]
	in order to contain all arrays of order $k$. By~\eqref{equation:binomialsimpleupperbound} and the definition of $f_d(q,k)$ we obtain
	\[\left(\frac{ef_d(q,k)}{k}\right)^{kd}\ge\binom{f_d(q,k)}{k}^d\ge q^{k^d},\]
	and thus we have
	\begin{align}\label{equation:lower}
		f_d(q,k)\ge \frac ke q^{k^{d-1}/d}.
	\end{align}
	In light of Theorem~\ref{theorem:universalwords-deterministic}, we know that for $d=1$ the lower bound obtained here is considerably far from being tight.
	But we will show that it is asymptotically tight for all $d\ge 2$.
	In fact, it is asymptotically tight for the random version of the problem.
	
	In order to prove Theorem~\ref{thm:2} we follow an approach taken by Alon~\cite{alon2017} in the study of universal graphs.
	Before diving into the proof let us first give a rough outline.
	
	Given $k\in\mathbb N$ sufficiently large and $n=(1+o(1))\frac ke q^{k^{d-1}/d}$,
	let $A\in \mathcal A_d(\Sigma,n)$ be a uniformly chosen $d$-array of order $n$ over $\Sigma$.
	For a fixed array $M\in \mathcal A_d(\Sigma,k)$, we consider the random variable $X$ that counts the number of copies of $M$ in~$A$.
	Since there are $q^{k^d}$ $d$-arrays of order $k$ over $\Sigma$, it is enough to prove that
	$\probability[X=0]=o(q^{-k^d})$ 
	and then use a union bound in order to conclude.
	However, is not easy to prove this directly.
	Instead, we consider the random variable $Y$ which is the size of the maximum family of disjoint copies of~$M$ in~$A$.
	It is clear that $Y=0$ if and only if $X=0$.
	Therefore, it is enough to estimate $\probability[Y=0]$.
	The random variable $Y$ has the advantage that it is 1-Lipschitz, meaning that changing the value of one entry of the random array may change the value of $Y$ in at most~$1$.
	Therefore, we may use (a known consequence of) Talagrand's inequality in order to upper bound $\probability[Y=0]$. 
	However, to be able to use this tool, we need estimates on the expected number of pairs of copies of $M$ in $A$ which overlap in some entries, which amounts to studying the variance of $X$.
	Grasping the asymptotic behaviour of this variance turns out to be the most technical part of our proof.
	
	\begin{theorem}[Talagrand's inequality~{\cite[Theorem 7.7.1]{AlonSpencer2016}}]
		Let $\Omega=\prod_{i\in [r]}\Omega_i$ be a product probability space, with the product probability measure,
		and let $h:\Omega\to \mathbb R$ be a 1-Lipschitz function, that is, $|h(x)-h(y)|\le 1$ when $x$ and $y$ differ in at most one coordinate.
		For $f\colon \mathbb N \rightarrow \mathbb N$, suppose that $h$ is $f$-certifiable, that is, if $x\in\Omega$ is such that $h(x)\ge s$ then there exists a set $I\subseteq [r]$ of size at most $f(s)$ such that if a vector $y\in \Omega$ coincides with $x$ on $I$, then $h(y)\ge s$.
		Then for $Y(x)=h(x)$ and all $b,t$, we have
		\[\probability[Y\le b-t\sqrt{f(b)}]\cdot \probability[Y\ge b]\le e^{-t^2/4}.\]
	\end{theorem}
	
	\begin{proof}[Proof of Theorem~\ref{thm:2}]
		Let $d,q, k\geqslant 2$ be fixed, $\eps > 0$.
		Let $n=(1+\eps)\frac{k}{e} q^{ \frac{k^{d-1}}{d}}$.
		Recall that we are interested in the asymptotic behaviour when $d, q$ are fixed and $k$ tends to infinity.
		In particular, we can assume whenever necessary that $k$ is sufficiently large with respect to $d, q, \eps$.
		All asymptotic notation is with respect to $k$ tending to infinity.
		
		Let $M\in \mathcal A_d(\Sigma,k)$ be a fixed $d$-array of order $k$ over the $q$-symbol alphabet $\Sigma$,
		and let $A$ be a uniformly chosen array from $\mathcal A_d(\Sigma,n)$.
		Our aim is to find a good upper bound on the probability that $A$ does not contain $M$, i.e., one allowing us to use a union bound to prove the result.
		
		Let $\mathcal T$ denote the collection of subsets of $[n]^d$ of the form $T=T_1\times\dotsb\times T_d$, where $|T_i|=k$ for each $1\le i\le d$.
		Given $T\in\mathcal T$, let $T(A)$ be the subarray of $A$ with entries $a_{\mathbf{i}}$ and $\mathbf{i} \in T$.
		Let $X_T$ be the indicator function of the event that $T(A)$ induces a copy of $M$, and let $X=\sum_{T\in \mathcal T}X_T$ be the number of copies of $M$ in $A$.
		Since for every $T\in\mathcal T$ we have $\expectation[X_T]=q^{-k^d}$, by linearity of the expectation then we have 
		\begin{equation}\label{equation:mu}
			\mu:=\expectation[X]=\binom{n}{k}^d q^{-k^d}=(1+o(1))(2\pi k)^{-d/2} (1+\eps)^{dk}\ge (16  \log q) k^{2d} ,
		\end{equation}
		where the last equality follows from the choice of $n$ and~\eqref{equation:stirling}, while the inequality follows from the assumption that $k$ is large.
		
		It will be crucial to show that we have
		\begin{equation}\label{equation:variance2}
			\variance(X)\le (1+o(1))\mu.
		\end{equation}
		To this end, we investigate (the expectation of) the random variable \[ Z = \sum_{T,T'} X_T X_{T'}, \]
		where the sum ranges over the pairs of distinct $T, T' \in \mathcal{T}$ which intersect in at least one cell. 
		For a vector $\mathbf{i} = (i_1,\dots,i_d) \in [k]^d$, we write  
		\[\Delta_{\mathbf{i}}=\sum_{T,T'\in\mathcal T_{\mathbf{i}}}\expectation[X_TX_{T'}],\]
		where $\mathcal T_{\mathbf{i}}$ denotes the collection of indices $T,T'\in\mathcal T$ such that $|T_j \cap T'_j| = i_{j}$ for all $j \in [d]$.
		Equivalently, $T$ and $T'$ intersect on exactly $i_j$ indices on the $j$-th coordinate.
		Therefore, if $\Delta = \expectation[Z]$, and $\mathbf{k} = (k, k, \dotsc, k)$ then we have
		\begin{align}
			\Delta= \sum_{T,T'} \expectation[ X_T X_{T'} ] = \sum_{\mathbf{i} \in [k]^d \setminus \{ \mathbf{k} \}}\Delta_{\mathbf{i}}.\label{equation:Delta}
		\end{align}
		Given $i\in[k]$, we define 
		\[\Lambda_i=\binom{n}{k}\binom{k}{i}\binom{n-k}{k-i} \hspace{.5cm}\text{and}\hspace{.5cm} L_d(i)=q^{\frac{i^{d}}{d}\left(1-({k}/{i})^{d-1}\right)}\frac1{(k-i)!}\binom{k}i\left(\frac{(1+\eps)k}e\right)^{k-i}. \]
		In order to prove \eqref{equation:variance2} we will use the following two claims.
		
		\begin{claim}\label{claim:Delta}
			For all $\mathbf{i} \in [k]^d$ we have
			\[\frac{\Delta_{\mathbf{i}}}{\mu}\le\prod_{j\in[d]}L_{d}(i_j).\]
		\end{claim}
		\noindent
		\emph{Proof of Claim~\ref{claim:Delta}.}
		Let $\mathbf{i} = (i_1, \dotsc, i_d)$.
		First, note that the total number of pairs $T, T'$ which intersect on $i_j$ entries on the $j$-th coordinate is exactly equal to $\prod_{j\in[d]}\Lambda_{i_j}$.
		Moreover, the union of two subarrays $T$ and $T'$ of this type together span exactly $2 k^d - i_1 \dotsb i_d$ cells.
		Then $X_T X_{T'} = 1$ holds if and only if in each one of those cells the correct symbol is attained, which implies
		\[ \Delta_{\mathbf{i}} \leq q^{-(2k^d-i_1 \dotsb i_d)} \prod_{j\in[d]}\Lambda_{i_j}. \]
		By the AM-GM inequality applied to $i_1^d, \dotsc, i_d^d$,
		we have $i_1 \dotsb i_d \le (\sum_{j=1}^d i_j^d) / d$.
		Thus we have
		\[\dfrac{\Delta_{\mathbf{i}}}{\mu} \leq \dfrac{q^{-(2k^d-i_1 \dotsb i_d)}\prod_{j\in[d]}\Lambda_{i_j}}{\binom{n}{k}^dq^{-k^d}}\le q^{-k^d+\frac{1}{d} \sum_{j\in[d]}i_j^d}\prod_{j\in[d]}\binom{k}{i_j}\binom{n}{k-i_j}.\]
		Using that $\tbinom{n}{k-i}\le n^{k-i}/(k-i)!$ and replacing $n=(1+\eps)\frac keq^{k^{d-1}/d}$ we have
		\[\begin{array}{ccl}\dfrac{\Delta_{\mathbf{i}}}{\mu}&\le&\displaystyle q^{-k^d+\frac1d\sum_{j\in[d]}i_j^d}\prod_{j\in[d]}\binom{k}{i_j}\dfrac{n^{k-i_j}}{(k-i_j)!}\\
			&\le & \displaystyle \prod_{j\in[d]}q^{\frac1d(i_j^d-k^d)}\frac1{(k-i_j)!}\binom{k}{i_j}\left(\frac{(1+\eps)k}e\right)^{k-i_j}q^{\frac1dk^{d-1}(k-i_j)}\\
			&=&\displaystyle \prod_{j\in[d]}L_{d}(i_j),
		\end{array} \]
		as desired.\hfill $\lrcorner$
		
		\begin{claim}\label{claim:correlations}
			If $1\le i\le k-1$, then $L_d(i)=o(k^{-d})$.
		\end{claim}
		\noindent
		\emph{Proof of Claim~\ref{claim:correlations}.}
		We will show that there exists $c > 0$ such that, for all $i$, $L_d(i) \leq e^{-ck}$ holds, which clearly yields the claim.
		
		Without loss of generality we may assume that $\eps\le (\log q)/8$, as otherwise we can restrict to a smaller array. Setting $i=k-j$ and by the Bernoulli inequality $(1 + j/(k-j))^{d-1} \ge 1 + j(d-1)/(k-j)$ we have
		\begin{align*}
			L_d(k-j) & = \displaystyle q^{\frac{(k-j)^d}d(1-(1+j/(k-j))^{d-1})}\frac{1}{j!}\binom{k}{j}\left(\frac{(1+\eps)k}e \right)^j \\
			& \le \displaystyle q^{-\frac{d-1}dj(k-j)^{d-1}}\frac{1}{j!}\binom{k}{j}\left(\frac{(1+\eps)k}e \right)^j
			\le \displaystyle q^{-\frac{d-1}dj(k-j)^{d-1}}\binom{k}{j}\left(\frac{(1+\eps)k}j \right)^j,
		\end{align*}
		where in the last step we used \eqref{equation:binomialsimpleupperbound} to bound $j!$.
		
		Taking logarithms,
		using $d \geq 2$,
		and later using $\log(1 + \eps) \leq \eps$, we obtain
		\begin{align}
			\log L_d(k-j)
			& \leq - \frac{1}{2} j (k-j) \log q + j \log \left( \frac{(1+\eps)k}{j} \right) + \log \binom{k}{j} \nonumber \\
			& \leq - \frac{1}{2} j (k-j) \log q + j ( \eps + \log k - \log j ) + \log \binom{k}{j} \label{equation:logLd}
		\end{align}
	Now, let $\beta \in (0,1/2)$ be small enough so that \[ \log\left(\frac{1}{1 - \beta}\right) + 2 \log(2) H(\beta) \leq \frac{1}{16} \log q \] holds.
		We now split the proof into two cases.
		Assume first that we have $j \leq (1 - \beta) k$.
		In this case, we use \eqref{equation:binomialsimpleupperbound} to bound the binomial coefficient and $k - j \geq \beta k$ in \eqref{equation:logLd} to get
		\begin{align*}
			\log L_d(k-j)
			& \leq - \frac{1}{2} j \beta k \log q + j ( \eps + \log k ) + j ( \log k + 1 - \log j ) \\
			& \leq j \left[ - \frac{1}{2} \beta k \log q + \eps + 2 \log k + 1 \right].
		\end{align*}
		For $k$ large enough, we have $\eps + 2 \log k + 1 \leq \frac{1}{4} \beta k \log q$,
		so the term in the brackets is negative.
		Using $j \geq 1$ we finally get $\log L_d(k-j) \leq - \frac{1}{4} \beta k \log q$,
		or equivalently $L_d(k-j) \leq \exp\left( - \frac{1}{4} \beta k \log q \right)$, which finishes the proof in this case.
		
		Now, consider the remaining case $j > (1 - \beta) k$.
		Since $\beta < 1/2$, we have $\binom{k}{j} \leq \binom{k}{(1 - \beta)k} = \binom{k}{ \beta k}$.
		Using \eqref{equation:entropy}, we get $\log \binom{k}{j} \leq (1 + o(1)) \log(2) H(\beta) k \leq 2 \log(2) H(\beta) k$.
		We plug this in \eqref{equation:logLd} to get 
		\begin{align*}
			\log L_d(k-j)
			& \leq - \frac{1}{2} (k-1) \log q + j ( \eps + \log k -\log j ) + 2 \log(2) H(\beta) k \\
			& \leq - \frac{1}{2} (k-1) \log q + j ( \eps - \log (1 - \beta ) ) + 2 \log(2) H(\beta) k \\
			& \leq k \left[ - \frac{1}{4} \log q + \eps - \log(1 - \beta) + 2 \log(2) H(\beta) \right],
		\end{align*}
		where in the last step we used $j \leq k$ and $(k-1)/2 \geq k/4$, which holds since $k$ is large.
		Now we use the assumption $\eps \leq (\log q) / 8$ and the choice of $\beta$ to get 
		\begin{align*}
			\log L_d(k-j)
			& \leq k \left[ - \frac{1}{8} \log q - \log(1 - \beta) + 2 \log(2) H(\beta) \right] \leq - k \frac{\log q}{16}.
		\end{align*}
		Thus in this case we have $L_d(k-j) \leq \exp( - k \frac{\log q}{16} )$, which finishes the proof of the claim.
		%
		\hfill $\lrcorner$
		\vskip .1 in
		
		Since the sum in~\eqref{equation:Delta} is over all the $k^d - 1$ many tuples $\mathbf{i}$ in $[k]^d$ distinct from $(k,\dots,k)$, then Claim~\ref{claim:Delta} and Claim~\ref{claim:correlations} together imply that
		\begin{align}
			\Delta =o(\mu).\label{equation:covariance}\end{align}
		Now, since $X$ is a sum of zero-one random variables, we have
		\[\variance(X) \le \expectation[X]+\sum_{T,T'\in\mathcal T} \covariance(X_T,X_{T'}).
		\]
		In the sum we only need to consider the pairs $T, T' \in \mathcal{T}$ with non-trivial intersection (otherwise the variables $X_T, X_{T'}$ are independent and thus their covariance is zero).
		Further, we have $\covariance(X_T,X_{T'})\le \expectation(X_TX_{T'})$.
		Therefore, by~\eqref{equation:covariance} we have
		\begin{align*}
			\variance (X)\le \mu+\Delta= (1+o(1))\mu,
		\end{align*}
		and so we have finally proved \eqref{equation:variance2}.
		
		By Chebyshev's inequality, and equations~\eqref{equation:mu} and~\eqref{equation:variance2} we have 
		\[\probability[|X-\mu|\ge \tfrac 14\mu]\le \frac{16\variance(X)}{\mu^2}\le \frac{32}{\mu}\to 0.\]
		Therefore, $X\ge 3\mu/4$ with probability at least $3/4$.
		Likewise, by Markov's inequality and~\eqref{equation:covariance} we have
		\[\probability[Z\ge \mu / 5] \le \frac{ 5 \expectation[Z]}{\mu} = \frac{5 \Delta}{\mu}\to 0,\]
		and therefore $Z\le \mu/4$ with probability at least $3/4$. In particular, both events hold at the same time with probability at least $1/2$. 
		
		Let $Y$ denote the random variable which is the size of the maximum family of disjoint copies of $M$ in $A$.
		Since $X\ge 3\mu/4$ and $Z\le \mu/4$ hold simultaneously with probability at least $1/2$, then, by conditioning on this event, we deduce that 
		\begin{equation}\label{equation:muhalf}
			\probability[Y\ge \mu/2]\ge 1/2.
		\end{equation} Notice also that $X = 0$ if and only if $Y = 0$.
		
		We are now ready to use Talagrand's inequality to finish the proof.
		Note that $h(A):=Y$ is $1$-Lipschitz, since by switching the value of one entry one can add or remove at most 1 copy of $M$ (the one using that entry).
		Moreover, $h(A)$ is $f$-certifiable for $f(s)=sk^d$.
		Using $b=\mu/2$ and $t=k^{-d/2}\sqrt{\mu/2}$, Talagrand's inequality and \eqref{equation:muhalf} give us
		\[\probability[X=0] = \probability[Y=0]\le 2e^{-\mu k^{-d}/8}.\]
		Finally, we use a union bound over all the possible choices of $M \in \mathcal A_d(\Sigma,k)$, to deduce that the probability that $A$ is not $k$-universal is at most
		\[2q^{k^d}e^{-\mu k^{-d}/8}\le 2q^{k^d}e^{- 2k^{d}\log q}=2q^{-k^{d}}\to 0,\]
		where the inequality comes from~\eqref{equation:mu}. The result follows.
	\end{proof}
	
	\begin{remark}
		The constant error term $\eps$ in Theorem~\ref{thm:2} can be improved to $\Omega(\log k / k)$.
		This can be seen by checking that replacing $\eps = C \log k / k$ (with $C$ being a large constant) is enough for \eqref{equation:mu} to hold. This does not change the rest of the calculations.
	\end{remark}
	
	\section{Concluding remarks}\label{seclast}
	
	\subsection{Uniqueness of subwords}
	One could wonder whether $k$-universal words can contain every subword exactly once, just as De Bruijn sequences contain every substring exactly once.
	However, this is not the case for essentially every value of~$k$. 
	
	\begin{proposition} Let $k\ge 2$ and $q\ge 2$ be integers, and $w$ a $k$-universal word over the $q$-symbol alphabet.
		There exists a word $u$ of length $k$ such that $w$ contains at least two copies of $u$.
	\end{proposition}
	
	
	\begin{proof}
		Since $w$ is $k$-universal, Theorem~\ref{theorem:universalwords-deterministic} implies $w$ has at least $qk$ symbols.
		Thus $w$ contains at least $\binom{qk}{k}$ subwords of length $k$.
		For all $1 \leq i < k$ we have $(qk-k+i)/i > q$ and thus
		\[ \binom{qk}{k} = \frac{qk (qk-1) \dotsb (qk - k + 1)}{k\cdot(k-1)!} = q \prod_{i=1}^{k-1} \frac{qk-k+i}{i} > q^k, \]
		so by the pigeonhole principle there is a word of length $k$ that appears at least twice as a subword in $w$.
	\end{proof}
	
	The same proof together with~\eqref{equation:stirling} yields that, for $q$ fixed and $k$ large, for every universal word $w$ there is a word of length $k$ contained at least $\binom{qk}{k} q^{-k} = 2^{\Omega(k)}$ times in $w$.
	
	Notice that the lower bound~\eqref{equation:lower}  for $f_d(q, k)$ when $d \ge 2$, follows from computing the size of a $k$-universal $d$-array containing every array of order $k$ exactly once. So perhaps studying the uniqueness of subarrays in universal $d$-arrays could help in obtaining even tighter bounds for these functions.  
	
%
	
	\subsection{Explicit universal $d$-arrays}
	
	In Theorem~\ref{theorem:universalwords-deterministic} we establish that $f_1(q,k)=qk$ by giving an explicit construction of a minimal universal word. 
	In contrast, in Theorem~\ref{thm:2} we establish the upper bound on $f_d(q,k)$ for $d\ge2$, by showing that random $d$-arrays of that order are likely to be $k$-universal. 
	It would be interesting to find ``explicit'' constructions of almost optimal universal arrays, for all relevant parameters.
	The simplest open question in this regard would be the following.
	\begin{question}
		 Is it possible to give an explicit description of a $k$-universal $2$-array on $2$ symbols of optimal (or close to optimal) order?
	\end{question}	
	In the setting of $d$-uniform hypergraphs, Alon and Sherman~\cite{alon2019} gave an explicit construction of $d$-graphs on $\Theta\left(2^{\binom{k}{d}/d}\right)$ vertices which contain every $k$-vertex $d$-graph as an induced subhypergraph,
	which is optimal up to the implied constant.

	\subsection{Universal random higher-dimensional permutations}
	For $\ell\ge 1$ an integer, let $S_\ell$ be the set of all permutations of $[\ell]$.
	For $n\ge k\ge 1$, we say that a permutation $\sigma\in S_n$ \emph{contains} a permutation $\tau\in S_k$ if there are indices $1\le i_1<\dots<i_k\le n$ such that for all $j,\ell \in [k]$,  $\sigma(i_j)<\sigma(i_\ell)$ if and only if $\tau(j)<\tau(\ell)$.
	A \emph{$k$-universal} permutation is one that contains all permutations of $S_k$.
	The question of the minimal $n$ such that there exists a $k$-universal permutation in $S_n$ was asked by Arratia~\cite{arratia}. Through a simple counting argument he observed that such an $n$ must satisfy $n\ge (1+o(1))k^2/e^2$ and conjectured that this was the optimal value of $n$ given $k$. The trivial bound which motivated this conjecture (now known to be false) was recently improved by
 Chroman, Kwan and Singhal~\cite{CKS2021}, although only by a small constant factor.
	
	The random version of this problem was posed by Alon (see~\cite{arratia}) who conjectured that a random permutation of order $(1+o(1))k^2/4$ is $k$-universal with high probability.
	If true, this bound would be tight, as can be deduced from the known results on the length of the longest increasing subsequence of random permutations.
	The best known upper bound for this problem is due to He and Kwan~\cite{xekwan}, who recently proved that a random permutation on $O(k^2\log\log k)$ elements is $k$-universal with high probability. 
	
	The study of higher dimensional permutations is ripe for further research.
	A \emph{line} of a $d$-array $A = (a_{i_1, \dotsc, i_d})$ of order $n$ is a sequence of elements obtained by choosing some $j \in [n]$ and looking at the entries $a_{i_1, \dotsc, i_{j-1}, \ell, i_{j+1}, \dotsc, i_d}$, for some fixed $i_1, \dotsc, i_{j-1}, i_{j+1}, \dotsc, i_d \in [n]$ and $\ell$ ranging from $1$ to $n$.
	Just as a usual permutation can be identified with a permutation matrix, it possible to define a \emph{$d$-dimensional permutation} (henceforth, \emph{$d$-permutation}) of order $n$ as a $(d+1)$-array of order $n$ over $\{0,1\}$, where each line contains a unique $1$ entry (see~\cite{LinialLuria2014, LinialSimkin2018} for equivalent definitions and discussion).
	
	Looking for connections with the case of permutations, we propose the following notion of ``universality'' for $d$-permutations.
	A \emph{$d$-pattern} of order $k$ is a sequence $(\sigma_1, \dotsc, \sigma_d)$  where $\sigma_\ell \in S_k$ for all $\ell\in [d]$.
	We say a $d$-permutation $M$ of order $n$ \emph{contains} a $d$-pattern of order $k$ if there exists a sequence $x^{(1)}, \dotsc, x^{(k)} \in [n]^{d+1}$ of index vectors such that $M_{x^{(i)}_1 x^{(i)}_2 \dotsb x^{(i)}_{d+1}} = 1$ for all $i \in [k]$, $x^{(1)}_1 < x^{(2)}_1 < \dotsb < x^{(k)}_1$
	(the first coordinates of the vectors are increasing), and further,
	for each $\ell \in [d]$ and all $i,j \in [k]$, it holds that $x^{(i)}_{\ell + 1} < x^{(j)}_{\ell + 1}$ if and only if $\sigma_\ell(i) < \sigma_\ell(j)$.
	Note that for $d = 1$ this is equivalent to the containment of one permutation in another.
	We say a $d$-permutation $M$ is \emph{$k$-pattern-universal} if it contains all $d$-patterns of order $k$.
	
	Linial and Simkin~\cite{LinialSimkin2018} considered ``monotone subsequences of length~$k$'' in $d$-permutations, which expressed in our language correspond to $d$-patterns of order~$k$ of the form $(\sigma, \dotsc, \sigma)$, where $\sigma$ is the identity function.
	They showed that the longest monotone subsequence in a random $d$-permutation of order $n$ has length $\Theta(n^{d/(d+1)})$ with high probability.
	This implies that a random $d$-permutation needs to have order at least $\Omega(k^{(d+1)/d})$ to be $k$-pattern-universal with high probability.
	In analogy with the case of permutations, we believe this to be tight.
	
	\begin{conjecture}For $d\ge 2$, there exists a constant $C>0$ such that a random $d$-permutation of order $Ck^{(d+1)/d}$ is $k$-pattern-universal with high probability as $k\to\infty$.	\end{conjecture}

	\section*{Acknowledgements}
	
	MPS was partially supported by ANID Doctoral scholarship ANID-PFCHA/ Doctorado Nacional/2017-21171132. MPS and DAQ thankfully acknowledge support from Programa Regional MATH-AMSUD, MATH190013; from Concurso para Proyectos de Investigaci\'on Conjunta CONICYT Chile -- FAPESP Brasil, 2019/13364-7; and from ANID + PIA/Apoyo a Centros Cient\'ificos y Tecnol\'ogicos de Excelencia con Financiamiento Basal AFB170001. DAQ thankfully acknowledges support from FONDECYT/ANID Iniciaci\'on en Investigaci\'on Grant 11201251. NSM was supported by the Czech Science Foundation, grant number GA19-08740S with institutional support RVO: 67985807.
	
	We also would like to thank the anonymous referees for their careful reading and helpful comments.
	\bibliographystyle{amsplain}
	\bibliography{arrays}

\end{document}